\documentclass[12pt]{amsart}
\usepackage{latexsym,amssymb,amsmath,tikz, xypic}

\numberwithin{equation}{section}
\newtheorem{theorem}{Theorem}[section]
\newtheorem{lemma}[theorem]{Lemma}

\newtheorem{corollary}[theorem]{Corollary}

\newtheorem{question}[theorem]{Question}

\theoremstyle{definition}
\newtheorem{definition}[theorem]{Definition} 
\newtheorem{construction}[theorem]{Construction}

\newtheorem{example}[theorem]{Example}

\begin{document}

 
\title{The regularity and $h$-polynomial of Cameron-Walker graphs}
\thanks{\today}

\author[T. Hibi]{Takayuki Hibi}
\address{Department of Pure and Applied Mathematics, Graduate School
of Information Science and Technology, Osaka University, Suita, Osaka
565-0871, Japan}
\email{hibi@math.sci.osaka-u.ac.jp}

\author[K. Kimura]{Kyouko Kimura}
\address{
Department of Mathematics, 
Faculty of Science, 
Shizuoka University, 836 Ohya, Suruga-ku, Shizuoka 422-8529, Japan}
\email{kimura.kyoko.a@shizuoka.ac.jp}

\author[K. Matsuda]{Kazunori Matsuda}
\address{Kitami Institute of Technology, Kitami, Hokkaido 090-8507, Japan}
\email{kaz-matsuda@mail.kitami-it.ac.jp}
 
\author[A. Van Tuyl]{Adam Van Tuyl}
\address{Department of Mathematics and Statistics\\
McMaster University, Hamilton, ON, L8S 4L8, Canada}
\email{vantuyl@math.mcmaster.ca}

 
\begin{abstract}
  Fix an integer $n \geq 1$, and consider the set of all connected finite simple
  graphs on $n$ vertices.  For each $G$ in this set, let
  $I(G)$ denote the edge ideal of $G$ in the polynomial ring
  $R = K[x_1,\ldots,x_n]$.  We initiate a study
  of the set  $\mathcal{RD}(n) \subseteq \mathbb{N}^2$ consisting of
  all the pairs $(r,d)$ where
  $r = {\rm reg}(R/I(G))$, the Castelnuovo-Mumford regularity,
  and $d = \deg h_{R/I(G)}(t)$, the degree of the
  $h$-polynomial, as we vary over all the connected graphs on $n$
  vertices. In particular, we identify sets $A(n)$ and $B(n)$
  such that $A(n) \subseteq \mathcal{RD}(n) \subseteq B(n)$.  When
  we restrict to the family of Cameron-Walker graphs on $n$ vertices,
  we can completely characterize all the possible $(r,d)$.
 \end{abstract}
 
\maketitle


\section{Introduction}

Let $R =K[x_1,\ldots,x_n]$ with $K$ a field, and let $I$ be a homogeneous
ideal of $R$.  In this paper we are interested in comparing
$r = {\rm reg}(R/I)$, the regularity of $R/I$, with $d = \deg h_{R/I}(t)$,
the degree
of the $h$-polynomial of $R/I$ (formal definitions
are postponed until the next section) for the class of edge ideals.
The first and third authors \cite{HM1,HM2} first showed that
for any integers $1 \leq r,d$, there exists a monomial ideal $I_{r,d}$
(and in fact, a lexsegment ideal)
such that ${\rm reg}(R/I_{r,d}) = r$ and $\deg h_{R/I_{r,d}}(t) =d$.
In collaboration with the last author \cite{HMVT}, it was later
shown that the ideal $I_{r,d}$ could in fact be an edge ideal.

Given these results, it may appear that there is no relationship between
the regularity and the degree of the $h$-polynomial, even in the case
that $I = I(G)$ is an edge ideal of a graph $G$.
However, our starting point is the following inequality
found in \cite[Theorem 13]{HMVT}; namely, if $G$ is a graph on
$n$ vertices, then
\begin{equation}\label{generalbound}
  {\rm reg}(R/I(G)) + \deg h_{R/I(G)}(t) \leq n,
\end{equation}
which gives a bound on the possible values of $r$ and $d$. 
If we fix an $n$ and compute
$(r,d) = ({\rm reg}(R/I(G)), \deg h_{R/I(G)}(t))$ for all connected graphs $G$ 
on $n = |V(G)|$ vertices, and plot the corresponding pairs, some interesting patterns
appear.  For example, Figure \ref{figure1} shows all the possible
$(r,d)$ for graphs on 8, respectively, 9 vertices.  In particular,
it is tantalizing to ask if the set of all possible $(r,d)$ for
a fixed $n$ can be described as the integer points of some convex
lattice polytope.
  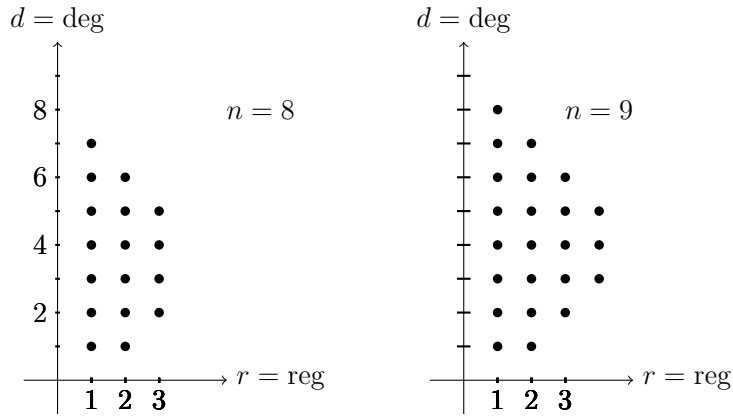
\begin{figure}
\begin{tikzpicture}[thick,scale=0.9, every node/.style={scale=0.9}]
\draw[thin,->] (-0.5,0) -- (2.5,0) node[right] {$r = {\rm reg}$};
\draw[thin,->] (0,-0.5) -- (0,5.0) node[above] {$d = \deg$};

\foreach \x [count=\xi starting from 0] in {1,2,3,4,5,6,7,8,9}{
        \draw (1pt,\x/2) -- (-1pt,\x/2);
    \ifodd\xi
        \node[anchor=east] at (0,\x/2) {$\x$};
    \fi
 \draw (1/2,1pt) -- (1/2,-1pt);
 \draw (2/2,1pt) -- (2/2,-1pt);
 \draw (3/2,1pt) -- (3/2,-1pt);
     \node[anchor=north] at (1/2,0) {$1$};
     \node[anchor=north] at (2/2,0) {$2$};
     \node[anchor=north] at (3/2,0) {$3$};
}

\foreach \point in {(1/2,1/2),(1/2,2/2),(1/2,3/2),(1/2,4/2),(1/2,5/2),
(1/2,6/2),(1/2,7/2),
(2/2,1/2),(2/2,2/2),(2/2,3/2),(2/2,4/2),(2/2,5/2),
(2/2,6/2),
(3/2,2/2),(3/2,3/2),(3/2,4/2),(3/2,5/2)}{
    \fill \point circle (2pt);
}

\draw[thin,->] (5.5,0) -- (8.5,0) node[right] {$r = {\rm reg}$};
\draw[thin,->] (6,-0.5) -- (6,5.0) node[above] {$d = \deg$};

\foreach \x [count=\xi starting from 0] in {1,2,3,4,5,6,7,8,9}{
    \draw (5.9,\x/2) -- (6.1,\x/2);
    \ifodd\xi
        \node[anchor=east] at (0,\x/2) {$\x$};
    \fi
\draw (6+1/2,1pt) -- (6+1/2,-1pt);
 \draw (6+2/2,1pt) -- (6+2/2,-1pt);
 \draw (6+3/2,1pt) -- (6+3/2,-1pt);
     \node[anchor=north] at (6+1/2,0) {$1$};
     \node[anchor=north] at (6+2/2,0) {$2$};
     \node[anchor=north] at (6+3/2,0) {$3$};
}

\
\foreach \point in {(6+1/2,1/2),(6+1/2,2/2),(6+1/2,3/2),(6+1/2,4/2),(6+1/2,5/2),(6+1/2,6/2),(6+1/2,7/2),(6+1/2,8/2),
(6+2/2,1/2),(6+2/2,2/2),(6+2/2,3/2),(6+2/2,4/2),(6+2/2,5/2),(6+2/2,6/2),(7,7/2),
(6+3/2,2/2),(6+3/2,3/2),(6+3/2,4/2),(6+3/2,5/2),(6+3/2,6/2),
(6+4/2,3/2),(6+4/2,4/2),(6+4/2,5/2)}{
    \fill \point circle (2pt);
}

\node at (3,4) {$n=8$};
\node at (8,4) {$n=9$};
\end{tikzpicture}
\caption{Possible $(r,d) = ({\rm reg}(R/I(G)),\deg h_{R/I(G)}(t))$
  for all connected graphs $G$ on 8 and 9 vertices}\label{figure1}
\end{figure}

To study this question, for each integer $n \geq 1$ we define:
\begin{eqnarray*}
\mathcal{RD}(n) & =& \left\{(r,d) ~\left|~
\begin{array}{c}
  \mbox{there exists a connected graph $G$ with $|V(G)|=n$} \\
  \mbox{and $(r,d) = ({\rm reg}(R/I(G)),\deg h_{R/I(G)}(t))$}
\end{array}
\right \}\right. \subseteq \mathbb{N}^2.
\end{eqnarray*}
One of our main results (see Theorem \ref{maintheorem})
describes
finite subsets $A(n), B(n) \subseteq \mathbb{N}^2$ such that
$A(n) \subseteq \mathcal{RD}(n) \subseteq B(n)$.  Both $A(n)$ and $B(n)$ are the
integer points of convex lattice polytopes.

Our results are stronger when we restrict to the connected graphs on
$n$ vertices that are also {\it Cameron-Walker graphs}.  Cameron-Walker
graphs are those graphs $G$ which satisfy the property that
the induced matching number of $G$ equals the matching number of $G$;
this family was first characterized by Cameron and Walker \cite{CW}.  From a
combinatorial commutative algebra point-of-view, these graphs
are attractive since ${\rm reg}(R/I(G))$ is also
equal to the induced matching number. In fact, a number of their
algebraic properties have been developed, e.g., see \cite{HHKO,HKMT}.
The following classification is one of our main results:

\begin{theorem}[Theorem \ref{maintheorem1}] Fix an $n \geq 5$.
  Then there exists a Cameron-Walker graph $G$ on $n$ vertices
  with ${\rm reg}(R/I(G)) = r$ and $\deg h_{R/I(G)}(t) = d$ if and only if
  \begin{enumerate}
  \item[$\bullet$] $2 \leq r \leq \lfloor \frac{n-1}{2} \rfloor$,
  \item[$\bullet$] $r \leq d \leq n -r $, and 
  \item[$\bullet$] $d \geq -2r+n+1$.
    \end{enumerate}
\end{theorem}
\noindent
The pairs $(r,d)$ in the above result form the
integer points of a convex lattice polytope.

Our paper is structured as follows.  In Section 2 we present the required
background, including the undefined terminology from the introduction.
In Section 3, we derive some properties about $\mathcal{RD}(n)$.
In Section 4, we introduce
Cameron-Walker graphs, and describe some of their relevant
homological invariants.
In Section 5, we give our proof to Theorem \ref{maintheorem1}.
This result is used to count the number of integer
points in the lattice polytope defined by Theorem \ref{maintheorem1}.
Our final section includes some questions and observations about
the ratio $|CW_{\mathcal{RD}}(n)|/|\mathcal{RD}(n)|$ as we vary $n$.
  
As a final comment, although our discussion
in this introduction
has been restricted to monomial ideals, some results are known about
the pairs $(r,d)$ for non-monomial ideals.  In particular, the
first and third authors \cite{HM3} showed that for all $2 \leq r \leq d$,
there is a binomial edge ideal $J_G$ with regularity $r$ and $h$-polynomial
of degree $d$; Kahle and Kr\"usemann \cite{KK} have
shown that for each integer $k \geq 0$, there exists a binomial
edge ideal $J_G$ with $r-d = k$.  
Finally, Favacchio, Keiper, and the last author \cite{FKVT} have
shown that if $4 \leq r \leq d$, there is a toric ideal
of a graph with regularity $r$ and $h$-polynomial with degree $d$.

\noindent
{\bf Acknowledgments.} 
Hibi, Kimura, and Matsuda's research was supported by JSPS KAKENHI 
19H00637, 15K17507, and 17K14165. 
Van Tuyl's research was supported by NSERC Discovery Grant 2019-05412 . 
This work was supported by the Research Institute for Mathematical Sciences, 
an International Joint Usage/Research Center located in Kyoto University. 

\section{Background}

In this section, we recall some of the relevant prerequisites about 
homological invariants, graph theory, and combinatorial commutative
algebra.  We have also include the formal definitions of the
undefined terms from the introduction.

\subsection{Homological Invariants}
Let $R = K[x_1, \ldots, x_n]$ denote the polynomial ring in $n$ variables
over a field $K$ with $\deg x_i = 1$ for all $i$.  For any ideal
$I$ of $R$, the {\it dimension} of $R/I$, denoted $\dim R/I$, is
the length of the longest chain of prime ideals in $R/I$.

If $I \subseteq R$
is a homogeneous ideal, then the {\it Hilbert series} of $R/I$ is
$$H_{R/I}(t) = \sum_{i \geq 0} \dim_K [R/I]_it^i$$
where $[R/I]_i$ denotes the $i$-th graded piece of $R/I$.
If $\dim R/I = d$, then 
the Hilbert series of $R/I$ is the form 
$$
H_{R/I}(t) = \frac{h_{0} + h_{1}t + h_{2}t^{2} + \cdots + h_{s}t^{s}}{(1 - t)^d}
= \frac{h_{R/I}(t)}{(1-t)^d}, 
$$
where each $h_{i} \in \mathbb{Z}$ (\cite[Proposition 4.4.1]{BH})
and $h_{R/I}(1) \neq 0$.
We say that 
$$
h_{R/I}(t) = h_{0} + h_{1}t + h_{2}t^{2} + \cdots + h_{s}t^{s}
$$ 
with $h_{s} \neq 0$ is the {\em $h$-polynomial} of $R/I$.  

The
({\em Castelnuovo-Mumford}) {\em regularity} of $R/I$,
 with $I$ homogeneous, is
$$
{\rm reg}(R/I)=  \max\{j-i ~|~ \beta_{i,j}(R/I) \neq 0\}
$$
where $\beta_{i,j}(R/I)$ denotes an $(i,j)$-th graded Betti number
in the minimal graded free resolution of $R/I$.  (For more details
see, for example,  \cite[Section 18]{Peeva}.)

\subsection{Graph theory}
Let $G = (V(G), E(G))$ be a finite simple graph
(i.e., a graph with no loops and no multiple edges) on the vertex set 
$V(G) = \{x_{1}, \ldots, x_{n}\}$ and edge set $E(G)$.

A subset $S \subset V(G)$ is an {\em independent set} of $G$ if
$\{x_{i}, x_{j}\} \not\in E(G)$ for all $x_{i}, x_{j} \in S$. 
In particular, the empty set $\emptyset$ is an independent set. 

A subset $\mathcal{M} \subset E(G)$ is a {\em matching} of $G$ 
if $e \cap e' = \emptyset$ for any $e, e' \in \mathcal{M}$ with $e \neq e'$. 
A matching $\mathcal{M}$ of $G$ is called an {\em induced matching} of $G$ if 
for $e, e' \in \mathcal{M}$ with $e \neq e'$, there is no edge $f \in E(G)$ 
with $e \cap f \neq \emptyset$ and $e' \cap f \neq \emptyset$. 
The {\em matching number} ${\rm m}(G)$ of $G$ is the maximum cardinality 
of the matchings of $G$. 
Similarly, the {\em induced matching number} ${\rm im}(G)$ of $G$
is the maximum cardinality  of the induced matchings of $G$. Because
an induced matching is also a matching, we always have
${\rm im}(G) \leq {\rm m}(G)$.

The  $S$-suspension (\cite[p.313]{HKM}) of a graph $G$ plays
an important role in our results;  we recall this construction.
If $G = (V(G),E(G))$ is a finite simple graph, then for
any independent set $S \subset V(G) = \{x_1,\ldots,x_n\}$,
we construct the graph $G^{S}$ with the vertex and the edge sets given by:
\begin{enumerate}
\item[$\bullet$] $V(G^{S}) = V(G) \cup \{x_{n + 1}\}$, where $x_{n + 1}$
  is a new vertex, and
\item[$\bullet$] $E(G^{S}) = E(G) \cup \left\{ \{x_{i}, x_{n + 1}\} ~|~ x_{i} \not\in S \right\}.$
  \end{enumerate}
That is, we add a new vertex $x_{n+1}$ and join it to every vertex {\it not}
in $S$.  The graph $G^{S}$ is called the {\em $S$-suspension} of $G$.  Note
that this construction still holds if $S = \emptyset$.

\subsection{Combinatorial commutative algebra}  Graphs can be studied
algebraically by employing the edge ideal construction.  If
$G=(V(G),E(G))$ is a finite simple graph on $V(G) = \{x_1,\ldots,x_n\}$,
we associate with $G$ the quadratic square-free monomial ideal
$$
I(G) = \langle x_{i}x_{j} ~|~ \{x_{i}, x_{j}\} \in E(G) \rangle \subseteq
R = K[x_1,\ldots,x_n]. 
$$
The ideal $I(G)$ is the {\it edge ideal} of the graph $G$.  We sometimes
write $K[V(G)]$ for the polynomial ring $K[x ~|~ x \in V(G)]$.

Under this construction, invariants of $G$ and homological invariants of
$I(G)$ are then related.  For example,
it is known that
$$\dim R/I(G) =  \max\left.\left\{ |S| ~\right|~ S \ \text{is an independent set of}\  G \right\}.$$  Another
relevant example of this behaviour is the following lemma.
\begin{lemma}\label{regbounds}
  For any finite simple graph $G = (V(G),E(G))$
  on $n$ vertices, we have
  $${\rm im}(G) \leq {\rm reg}(R/I(G)) \leq {\rm m}(G)
  \leq \left\lfloor \frac{n}{2} \right\rfloor.$$
\end{lemma}

\begin{proof}
  The first inequality is \cite[Lemma 2.2]{Katzman}, and the second
  inequality is \cite[Theorem 6.7]{HVT}.  The last inequality follows from
  the observation that ${\rm m}(G)$ edges in $G$ contain $2{\rm m}(G)$
  distinct vertices, so $2{\rm m}(G) \leq n$.
  \end{proof}

If $G$ is a graph with an $S$-suspension $G^S$, then by
virtue of \cite[Lemma 1.5]{HKM}, we have
some relationships between the homological invariants of $I(G)$
and $I(G^S)$.  

\begin{lemma}\label{S-suspension}
  Let $G$ be a finite simple graph on $V(G) = \{x_1,\ldots,x_n\}$, and
  suppose that $G^S$ is the $S$-suspension of $G$ for some independent
  set $S$ of $V(G)$.  If $I(G) \subseteq R =K[x_1,\ldots,x_n]$ and
  $I(G^S) \subseteq  R' = K[x_1,\ldots,x_n,x_{n+1}]$
  are the respective edge ideals,
  then
  \begin{enumerate}
	\item ${\rm reg}(R'/I(G^{S})) = {\rm reg}(R/I(G))$ if $G$ has no isolated vertices. 
	\item 
  \begin{displaymath}
    H_{R'/I(G^S)} (t) = H_{R/I(G)} (t) + \frac{t}{(1-t)^{|S|+1}}. 
  \end{displaymath}
  In particular, 
  $\deg h_{R'/I(G^{S})}(t) = \deg h_{R/I(G)}(t)$ if $|S| = \dim R/I(G) - 1$. 
	\item $\dim R'/I(G^{S}) = \dim R/I(G)$ if $|S| \leq \dim R/I(G) - 1$. 
\end{enumerate} 
\end{lemma}

Let $H_{1}$ and $H_{2}$ be finite simple graphs, and let
$H = H_{1} \cup H_{2}$ the disjoint union of $H_1$ and $H_2$. 
Then one has the following identities.

\begin{lemma}\label{disjoint_union}
Under the above situation, we have 
\begin{enumerate}
	\item ${\rm reg}(K[V(H)]/I(H)) = {\rm reg}(K[V(H_{1})]/I(H_1))+{\rm reg}(K[V(H_{2})]/I(H_2))$. 
	\item $\deg h_{K[V(H)]/I(H)}(t) = \deg h_{K[V(H_{1})]/I(H_1)}(t) + \deg h_{K[V(H_{2})]/I(H_2)}(t)$.
\end{enumerate}
\end{lemma}
\begin{proof}
The result follows from the fact that $K[V(H)]/I(H)$ is the tensor product of $K[V(H_{1})]/I(H_1)$ and $K[V(H_{2})]/I(H_2)$. 
\end{proof}


\section{Properties of the set $\mathcal{RD}(n)$}

Recall from the introduction
that for each $n \geq 1$, the set $\mathcal{RD}(n)$
compares the regularity and the degree of the $h$-polynomial over
all connected graphs on $n$ vertices.  The purpose
of this section is to derive some basic properties of this set.
We begin with the following observations, which relies heavily on
the $S$-suspension construction.

\begin{lemma}\label{BasicLemma}
  For all $n \geq 1$, we have $\mathcal{RD}(n) \subseteq \mathcal{RD}(n+1)$.
\end{lemma}

\begin{proof}
  Let $(r, d) \in \mathcal{RD}(n)$. Then there exists a
  connected graph $G$ with $n$ vertices such that
  ${\rm reg}(R/I(G)) = r$ and ${\rm deg} h_{R/I(G)}(t) = d$. Take an
  independent set $S$ of $G$ with $|S| = \dim R/I(G) - 1$.  This is possible
  since there is an independent set $W$ with $|W| = \dim R/I(G)$, so
  we can take $S = W\setminus \{w\}$ for any $w \in W$.  By virtue of Lemma \ref{S-suspension} (1) and (2), we have
  ${\rm reg}(R'/I(G^{S})) = r$ and ${\rm deg} h_{R'/I(G^{S})}(t) = d$. Since
  $G^S$ is a graph on $n+1$ vertices, we have $(r, d) \in \mathcal{RD}(n + 1)$.
 \end{proof}

\begin{lemma}\label{BasicLemma2}
Let $n_{1}, \ldots, n_{p} \geq 2$ be integers. 
Suppose that $(r_{i}, d_{i}) \in \mathcal{RD}(n_{i})$ for all $i = 1, \ldots, p$. 
Then $(r_{1} + \cdots + r_{p}, d_{1} + \cdots + d_{p}) \in \mathcal{RD}(n_{1} + \cdots + n_{p} + 1)$. 
\end{lemma}
\begin{proof}
Let $G_i$ denote a  connected graph  with $n_{i}$ vertices such that 
${\rm reg}(K[V(G_{i})]/I(G_{i})) = r_{i}$ and ${\rm deg} h_{K[V(G_{i})]/I(G_{i})}(t) = d_{i}$ for all $i = 1, \ldots, p$. 
Let us consider the disjoint union $G = G_{1} \cup \cdots \cup G_{p}$. 
By virtue of Lemma \ref{disjoint_union}, one has ${\rm reg}(K[V(G)]/I(G)) = r_{1} + \cdots + r_{p}$ and ${\rm deg} h_{K[V(G)]/I(G)}(t) = d_{1} + \cdots +d_{p}$. 
Let $S \subset V(G)$ be an independent set of $G$
with $|S| = \dim K[V(G)]/I(G) - 1$. 
Then the $S$-suspension $G^{S}$ has $n_{1} + \cdots + n_{p} + 1$ vertices and 
${\rm reg}(K[V(G^{S})]/I(G^{S})) = r_{1} + \cdots + r_{p}$ and ${\rm deg} h_{K[V(G^{S})]/I(G^{S})}(t) = d_{1} + \cdots + d_{p}$ 
by Lemma \ref{S-suspension}. 
Hence we have the desired conclusion. 
\end{proof}

We now focus on the lattice points of $\mathcal{RD}(n)$. Our starting
point is the next lemma which identifies
some lattice points of this set.  To prove
this lemma, we require the following two graphs.  The ribbon
graph, denoted $G_{\rm ribbon}$, is the graph on five vertices
as given in Figure \ref{fig:G3}.
\begin{figure}[htbp]
\centering

\begin{xy}
	\ar@{} (0,0);(50,0) *{\text{$G_{\rm ribbon} =$}};
	\ar@{} (0,0);(90, 0)  *++!D{x_5} *\cir<2pt>{} = "A";
	\ar@{-} "A";(70, 8)   *++!D{x_1} *\cir<2pt>{} = "B";
	\ar@{-} "A";(70, -8)  *++!U{x_2} *\cir<2pt>{} = "C";
	\ar@{-} "A";(110, 8)  *++!D{x_3}  *\cir<2pt>{} = "D";
	\ar@{-} "A";(110, -8) *++!U{x_4} *\cir<2pt>{} = "E";
	\ar@{-} "B";"C";
	\ar@{-} "D";"E";
\end{xy}

\caption{The graph $G_{\rm{ribbon}}$}
\label{fig:G3}
\end{figure}
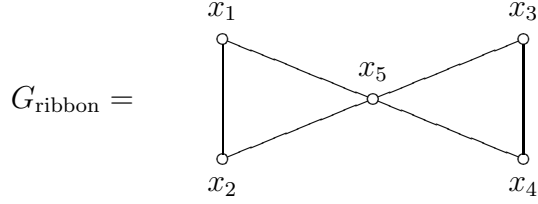
The regularity and the degree of the $h$-polynomial for
$K[V(G_{\rm ribbon})]/I(G_{\rm ribbon})$ 
are computed in \cite[Example 10]{HMVT} (or can be computed via a computer algebra system):
$${\rm reg}(K[V(G_{\rm ribbon})]/I(G_{\rm ribbon})) = 2
~\mbox{and}~~ \deg h_{K[V(G_{\rm ribbon})] / I(G_{\rm ribbon})}(t) = 1.$$ 

Our second family is $D_r$, where $D_r$ is a graph on $2r$ vertices
consisting of the disjoint union of $r$ paths of length $1$.
In this case $I(D_r)$ is a complete intersection since
$I(D_r) = \langle x_1x_2,x_3x_4,\ldots,x_{2r-1}x_{2r}\rangle$ is generated
by $r$ monomials which have pairwise disjoint support.  So, by
properties of complete intersections,
$$H_{K[V(D_r)]/I(D_r)}(t) =  \frac{(1+t)^r}{(1-t)^r},$$
  and consequently,
  $h_{K[V(D_r)]/I(D_r)}(t) = (1 + t)^{r}$  and $\dim K[V(D_r)]/I(D_r) = r$.
  Moreover, since the Koszul complex gives a minimal
  free resolution of $K[V(D_r)]/I(D_r)$, we have ${\rm reg}(K[V(D_r)]/I(D_r)) =r$.

\begin{lemma}\label{RD}
  Let $r \geq 1$, $d \geq 1$ be integers. 
  \begin{enumerate}
    \item Then $(r,1) \in \mathcal{RD}(2^r+r-1)$.
  \item If $r < d$, then $(r, d) \in \mathcal{RD}(r + d)$. 
  \item If $r \geq 2$, then $(r, d) \not\in \mathcal{RD}(2r)$. 
  In particular, if $(r, d) \in \mathcal{RD}(n)$, then $r \leq \lfloor \frac{n - 1}{2} \rfloor$. 
  \item If $r = d \geq 2$, then $(r, d) = (r,r) \in \mathcal{RD}(2r+1)$. 
  \item If $r = d+1$ and $r$ is even (resp.\  $r$ is odd), 
    then $(r,d) = (r, r - 1) \in \mathcal{RD}(2r + 1)$ 
    (resp.\  $(r,d) = (r, r - 1) \in \mathcal{RD}(2r + 2)$). 
  \item Let $c$ be an integer with $c \geq 1$.  
    If $r \geq d + 2$ and $cd < r \leq (c + 1)d$, then $(r, d)
    \in \mathcal{RD}\left( (2^{c} + 1)r - ((c - 1)2^{c} + 1)d + 1 \right)$. 
  \end{enumerate}
\end{lemma}

\begin{proof}  Statement (1) follows from \cite[Lemma 12]{HMVT}
  which constructs a connected graph $G$ on
  $2^r+r-1$ vertices that has ${\rm reg}(K[V(G)]/I(G)) =r$ and $\deg h_{K[V(G)]/I(G)}(t) =1$.

  To prove (2), let $D_r$ be the graph
  defined prior to this lemma.  Let $S_1$ be an independent
  set of $D_r$ with $|S_1| = r$ (for example, take one vertex from
  each path of length one).   The $S$-suspension
  graph $B_1 = D_r^{S_1}$ has $2r+1$ vertices,  and by Lemma
  \ref{S-suspension} (1) ${\rm reg}(K[V(B_1)]/I(B_1)) = r$
  and by Lemma \ref{S-suspension} (2)
  $$
  H_{K[V(B_1)]/I(B_1)}(t) =
  \frac{(1+t)^r}{(1-t)^r} + \frac{t}{(1-t)^{r+1}}
  = \frac{(1+t)^r(1-t) + t}{(1-t)^{r+1}}
  $$
  and so $\deg h_{K[V(B_1)]/I(B_1)}(t) = r + 1$.

  We now reiterate this process.  Let $S_i$ be the independent
  set of $B_{i-1}$ of size $r+{i-1}$ that contains the $r$ independent
  elements of $S_1$ and $y_1,\ldots,y_{i-1}$ where $y_j$ was the
  new vertex we added when we constructed $B_j = B_{j-1}^{S_{j-1}}$ by
  forming the $S$-suspension of $B_{j-1}$ with $S_{j-1}$.  Each set
  $S_i$ is independent because each new $y_j$ is only adjoined
  to the vertices not in $S_1$ in $D_r$.  By induction on $i$, Lemma
  \ref{S-suspension} implies that the graph $B_{i}$ satisfies
  ${\rm reg}(K[V(B_i)]/I(B_i)) = r$ and
  $\deg h_{K[V(B_i)]/I(B_i)} = r+i$.   It then follows that
  $B_{d-r}$ has $2r+d-r = r+d$ vertices, ${\rm reg}(K[V(B_{d-r})]/I(B_{d-r})) = r$
  and  $\deg h_{K[V(B_{d-r})]/I(B_{d-r})}(t) = d$.  So $(r,d) \in \mathcal{RD}(r+d)$. 
  
  For the proof of $(3)$, we assume that $(r,d) \in \mathcal{RD}(2r)$. Then there exists a connected simple graph $G$ with 
  ${\rm reg}(K[V(G)]/I(G)) = r \geq 2$ and $|V(G)| = 2r$.
  By Lemma \ref{regbounds}, we have
  $r = {\rm reg}(K[V(G)]/I(G)) \leq  {\rm m}(G) \leq \lfloor \frac{2r}{2} \rfloor$, that is, ${\rm reg}(K[V(G)]/I(G)) = {\rm m}(G) = r$. 
  If ${\rm im}(G) = r$, then $G = D_{r}$, a contradiction for the connectivity of $G$. 
  Hence ${\rm im}(G) < {\rm reg}(K[V(G)]/I(G)) = {\rm m}(G)$. Then \cite[Theorem 11]{TNT} says that $G$ is a pentagon, 
  but this is a contradiction. Thus $(r,d) \not\in \mathcal{RD}(2r)$. 

  For the proof of $(4)$, again consider the graph $D_r$, and
  let $S$ be an independent
  set with $|S| = r-1 = \dim K[V(D_r)]/I(D_r) -1$.  Then by Lemma \ref{S-suspension}
  (1) and (2),
  the ring $K[V(D_r^S)]/I(D_r^S)$ has regularity $r$ and
  $\deg h_{K[V(D_r^S)]/I(D_r^S)}(t) = \deg h_{K[V(D_r)]/I(D_r)]}(t) = r$.
  Since $D_r^S$ has $2r+1$ vertices, $(r,r) \in \mathcal{RD}(2r+1)$. 

  To prove (5), first assume that $r$ is even. Let $D_r$
  be as above, and consider the $S$-suspension with $S = \emptyset$.
  By Lemma \ref{S-suspension} (1), the regularity of
  $K[V(D_r^\emptyset)]/I(D_r^\emptyset)$ equals $r$, while
  $$H_{K[V(D_r^\emptyset)]/I(D_r^\emptyset)}(t) =
  H_{R/I(D_r)}(t)+ \frac{t}{1-t} =  \frac{(1+t)^r}{(1-t)^r} + \frac{t}{1-t}
  = \frac{(1+t)^r+t(1-t)^{r-1}}{(1-t)^r}.$$
   Because $r$ is even, when we simplify the $h$-polynomial we find
   $\deg h_{K[D_r^\emptyset]/I(D_r^\emptyset)}(t) = r-1$.  So
   $(r,r-1) \in \mathcal{RD}(2r+1)$.
   
    If we instead assume that $r$ is odd, 
    consider the graph $G$ which is the disjoint union of
    $D_{r-1}^\emptyset$ and $D_1$.
    Then $|V(G)|=2r+1$, ${\rm reg}(K[V(G)]/I(G)) = \dim K[V(G)]/I(G) = r$, 
    and $\deg h_{K[V(G)]/I(G)}(t) = r-1$. 
    Let $S$ be an independent set of $G$ with 
    $|S| = r-1$.  By  Lemma \ref{S-suspension}
    the $S$-suspension of $G$
    creates a graph 
    with $(r, r - 1) \in \mathcal{RD}(2r + 2)$. 
    
    Finally, we give a proof of (6).
    We set $i= r-d (\geq 2)$. 
    Note that
    $$(r, r - i) = (cr - (c + 1)i) \cdot (c, 1) + (ci - (c - 1)r) \cdot (c + 1, 1).$$ 
    By virtue of (1), one has $(c, 1) \in \mathcal{RD}(2^{c} + c - 1)$ and $(c + 1, 1) \in \mathcal{RD}(2^{c + 1} + c)$. 
    Then, since $i = r - d$, it follows that 
    \begin{eqnarray*}
    (r,d) = (r, r - i) &\in& \mathcal{RD}\left( (cr - (c + 1)i)(2^{c} + c - 1) + (ci - (c - 1)r)(2^{c + 1} + c) + 1 \right) \\
    &=&  \mathcal{RD}\left( (2^{c} + 1)r - ((c - 1)2^{c} + 1)d + 1 \right)
    \end{eqnarray*}
    by virtue of Lemma \ref{BasicLemma2}. 
    We now have the desired conclusion. 
\end{proof}

By virtue of Lemmas \ref{BasicLemma} and \ref{RD}, we have the following theorem.
Recall that  if $(r,d) \in \mathcal{RD} (n)$, 
then $r \leq \lfloor \frac{n - 1}{2} \rfloor$ by Lemma \ref{RD} (3) and
$r + d \leq n$ by \eqref{generalbound}.

\begin{theorem}\label{LatticePointsRD}
  Let $r \geq 1$, $d \geq 1$, and $n \geq 3$ be integers. 
  Assume that $r \leq \lfloor \frac{n - 1}{2} \rfloor$ and $r + d \leq n$.  Then 
  \begin{enumerate}
  \item If $r < d$, then $(r, d) \in \mathcal{RD}(n)$. 
  \item If $r=d \geq 2$ and $r+d = r+r < n$, then $(r,d) = (r, r) \in \mathcal{RD} (n)$. 
  \item If $r = d + 1$ and $r < \lfloor \frac{n-1}{2} \rfloor$, then $(r, d) = (r, r - 1) \in \mathcal{RD} (n)$. 
  \end{enumerate}
\end{theorem}

\begin{proof}
  For the proof of (1), assume that $r < d$. Since $r + d \leq n$, we have 
  $(r, d) \in \mathcal{RD} (r + d) \subseteq \mathcal{RD} (n)$ by virtue of Lemmas \ref{BasicLemma} and \ref{RD} (2). 
   Statement (2) follows from Lemmas \ref{BasicLemma} and \ref{RD}(4).

   For statement $(3)$, since $r < \frac{n-1}{2}$, we have $2r+1 < n$,
  or equivalently, $2r+2 \leq n$.  If $r = d+1$, then by
  Lemma \ref{RD} (5),
  we have $(r,r-1) \in \mathcal{RD}(2r+1)$ or $\mathcal{RD}(2r+2)$,
  depending upon the parity of $r$.  The result now follows
  from Lemma \ref{BasicLemma} since $2r+1 < 2r+2 \leq n$.
\end{proof}

For a positive integer $n$, we define 
\begin{eqnarray*}
\displaystyle A(n) &=& \left\{ (r, d) \ \middle| \ 1 \leq r < \left\lfloor \frac{n - 1}{2} \right\rfloor, \, 1 \leq d \leq n - r, \, r - d \leq 1 \right\}, \\
\displaystyle B(n) &=& \left\{ (r, d) \ \middle| \ 1 \leq r \leq \left\lfloor \frac{n - 1}{2} \right\rfloor, \, 1 \leq d \leq n - r \right\}. 
\end{eqnarray*}  
Both $A(n)$ and $B(n)$ are the integer points of a convex lattice polytopes. 
The following theorem is one of our main theorem, and it follows
directly from Theorem \ref{LatticePointsRD}.

\begin{theorem}\label{maintheorem}
Let $n \geq 3$ be an integer. 
Let $A(n)$ and $B(n)$ be sets of integer points as above. Then 
\[
A(n) \subseteq  \mathcal{RD}(n) \subseteq B(n). 
\]
\end{theorem}

\begin{proof}
  For all $(r,d) \in A(n)$ except $(r,d) = (1,1)$, the first inclusion
  follows from Theorem \ref{LatticePointsRD}.  For $(1,1)$,
  note that the graph $D_1$ has ${\rm reg}(R/I(D_1)) = \deg h_{R/I(D_1)}(t) =1$. 
  So by Lemma \ref{BasicLemma}, one has $(1,1) \in \mathcal{RD}(n)$.
  The second inclusion follows from Lemma \ref{RD} (3) and
  \eqref{generalbound}. 
  \end{proof}

We end this section with a question inspired
by our results and computer experiments.

\begin{question}
  For all $n \geq 1$, is the set $\mathcal{RD}(n)$ a convex set?
  That is, if $(r,d)$ and $(r,d')$ with $d<d'$,
  respectively $(r',d)$ with $r<r'$,
  are in $\mathcal{RD}(n)$,
  is $(r,s) \in \mathcal{RD}(n)$ for all $d < s < d'$, respectively
  is $(s,d) \in \mathcal{RD}(n)$ for all $r < s < r'$?
\end{question}


\section{Cameron-Walker graphs: relevant properties}

For the remainder of this paper we will focus on describing
all possible pairs $(r,d) = ({\rm reg}(R/I(G)),\deg h_{R/I(G)}(t))$
when $G$ is a Cameron-Walker graph.  Towards this end, we
introduce the following subset of $\mathcal{RD}(n)$:
\[
CW_\mathcal{RD}(n)  = \left\{(r,d) ~\left|~
\begin{array}{c}
  \mbox{there exists a Cameron-Walker graph $G$ with} \\
  \mbox{$|V(G)| =n$ and $(r,d) = ({\rm reg}(R/I(G)),\deg h_{R/I(G)}(t))$}
\end{array}
\right \}\right..
\]
In this section we review the relevant background on Cameron-Walker
graphs so that in the next section we can completely
describe $CW_\mathcal{RD}(n)$ for all $n \geq 1$.

Recall from Lemma \ref{regbounds} the following inequalities:
\[
{\rm im}(G) \leq {\rm reg} \left(K[V(G)]/I(G)\right) \leq {\rm m}(G).
\]
By virtue of \cite[Theorem 1]{CW} together with \cite[Remark 0.1]{HHKO}, 
we have that the equality ${\rm im}(G) = {\rm m}(G)$ holds 
if and only if $G$ is one of the following graphs: 
\begin{itemize}
\item a star graph, i.e., a graph joining some paths of length
  $1$ at one common vertex (see Figure \ref{fig:Star}); 
\item a star triangle, i.e., a graph joining some triangles at
  one common vertex (see Figure \ref{fig:Star}); or
\item a finite graph consisting of a connected
  bipartite graph with vertex partition
  $\{v_{1}, \ldots, v_{m}\} \cup \{w_{1}, \ldots, w_{p}\}$
  such that there is at least one leaf edge
  attached to each vertex $v_{i}$ and that there may be possibly
  some pendant triangles attached to each vertex $w_{j}$;
  see Figure \ref{fig:CameronWalkerGraph} where $s_{i} \geq 1$
  for all $i = 1, \ldots, m$ and $t_{j} \geq 0$ for all $j = 1, \ldots, p$. 
  Note that a leaf edge is an edge meeting a vertex of degree $1$ and a
  pendant triangle is a triangle where two vertices have degree $2$ and the
  remaining vertex has degree more than $2$. 
\end{itemize} 

\begin{figure}[htbp]
  \centering
\bigskip

\begin{xy}
	\ar@{} (0,0);(55, -16)   *\cir<4pt>{} = "C"
	\ar@{-} "C";(35, 0)  *\cir<4pt>{} = "D";
	\ar@{-} "C";(45, 0)  *\cir<4pt>{} = "E";
	\ar@{} "C"; (55, 0) *++!U{\cdots}
	\ar@{-} "C";(75, 0)  *\cir<4pt>{} = "F";
	\ar@{} (0,0);(105, 0)  *\cir<4pt>{} = "C2"
	\ar@{-} "C2";(85, -12)  *\cir<4pt>{} = "D2";
	\ar@{-} "C2";(89, -16)  *\cir<4pt>{} = "E2";
	\ar@{-} "D2";"E2";
	\ar@{} "C2"; (105, -5) *++!U{\cdots}
	\ar@{-} "C2";(125, -12)  *\cir<4pt>{} = "F2";
	\ar@{-} "C2";(121, -16)  *\cir<4pt>{} = "G2";
	\ar@{-} "F2";"G2"; 
\end{xy}

\bigskip

\caption{The star graph (left) and the star triangle (right)}
\label{fig:Star}
\end{figure}
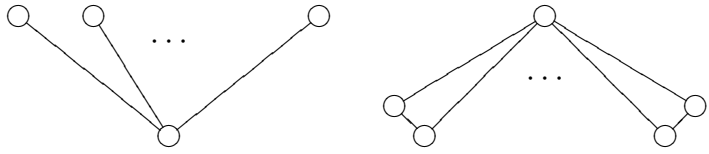

\begin{definition}
A finite connected simple graph $G$ is {\em a Cameron-Walker graph} 
if ${\rm im}(G) = {\rm m}(G)$ and if $G$ is neither a star
graph nor a star triangle. 
\end{definition}

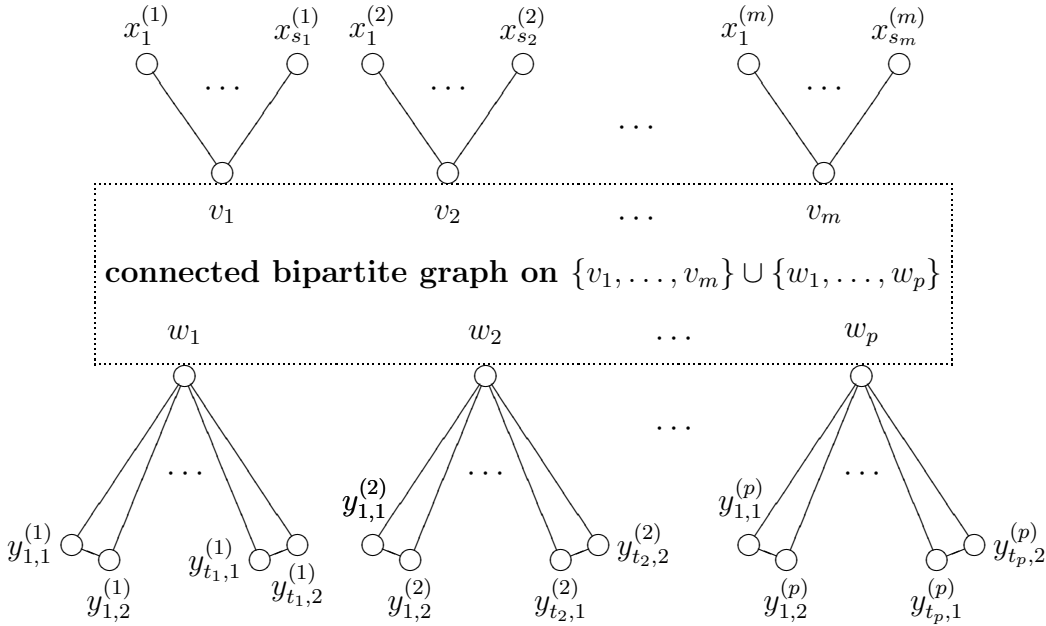
\begin{figure}[htbp]
  \centering
\begin{xy}
	\ar@{} (0,0);(18, 12)  = "A";
	\ar@{.} "A";(132, 12)  = "B";
	\ar@{.} "A";(18, -12)  = "C";
	\ar@{.} "B";(132, -12)  = "D";
	\ar@{.} "C";"D";
	\ar@{} (0,0);(75,0) *{\text{\bf{connected bipartite graph on} $\{v_{1}, \ldots, v_{m}\} \cup \{w_{1}, \ldots, w_{p}\}$}};
	\ar@{} "A";(35, 13.5)  *\cir<4pt>{} = "E1";
	\ar@{-} "E1";(25, 28) *++!D{x^{(1)}_{1}} *\cir<4pt>{};
	\ar@{-} "E1";(45, 28) *++!D{x^{(1)}_{s_{1}}} *\cir<4pt>{};
	\ar@{} (0,0); (35, 28) *++!U{\cdots}
	\ar@{} "A";(65, 13.5)  *\cir<4pt>{} = "E2";
	\ar@{-} "E2";(55, 28) *++!D{x^{(2)}_{1}} *\cir<4pt>{};
	\ar@{-} "E2";(75, 28) *++!D{x^{(2)}_{s_{2}}} *\cir<4pt>{};
	\ar@{} (0,0); (65, 28) *++!U{\cdots}
	\ar@{} "A";(115, 13.5)  *\cir<4pt>{} = "E";
	\ar@{-} "E";(105, 28) *++!D{x^{(m)}_{1}} *\cir<4pt>{};
	\ar@{-} "E";(125, 28) *++!D{x^{(m)}_{s_{m}}} *\cir<4pt>{};
	\ar@{} (0,0); (115, 28) *++!U{\cdots}
	\ar@{} (0,0);(35,8) *{\text{$v_{1}$}};
	\ar@{} (0,0);(65,8) *{\text{$v_{2}$}};
	\ar@{} (0,0);(90,4) *++!D{\cdots};
        \ar@{} (0,0);(115,8) *{\text{$v_{m}$}};
	\ar@{} (0,0);(90,16) *++!D{\cdots};
	\ar@{} "E1";(30, -13.5)  *\cir<4pt>{} = "F10";
	\ar@{-} "F10";(15, -36) *++!R{y^{(1)}_{1,1}} *\cir<4pt>{} = "F11";
	\ar@{-} "F10";(20, -38) *++!U{y^{(1)}_{1,2}} *\cir<4pt>{} = "F12";
	\ar@{-} "F11";"F12";
	\ar@{-} "F10";(40, -38) *++!R{y^{(1)}_{t_{1},1}} *\cir<4pt>{} = "Ft1";
	\ar@{-} "F10";(45, -36) *++!U{y^{(1)}_{t_{1},2}} *\cir<4pt>{} = "Ft2";
	\ar@{-} "Ft1";"Ft2";
	\ar@{} "E1";(70, -13.5)  *\cir<4pt>{} = "F20";
	\ar@{-} "F20";(55, -36) *\cir<4pt>{} = "F21";
	\ar@{} (0,0);(54,-30) *{\text{$y^{(2)}_{1,1}$}};
	\ar@{-} "F20";(60, -38) *++!U{y^{(2)}_{1,2}} *\cir<4pt>{} = "F22";
	\ar@{-} "F21";"F22";
	\ar@{-} "F20";(80, -38) *++!U{y^{(2)}_{t_{2},1}} *\cir<4pt>{} = "Ftt1";
	\ar@{-} "F20";(85, -36) *++!L{y^{(2)}_{t_{2},2}} *\cir<4pt>{} = "Ftt2";
	\ar@{-} "Ftt1";"Ftt2";
	\ar@{} "E1";(120, -13.5)  *\cir<4pt>{} = "Fn0";
	\ar@{-} "Fn0";(105, -36) *\cir<4pt>{} = "Fn1";
	\ar@{} (0,0);(104,-30) *{\text{$y^{(p)}_{1,1}$}};
	\ar@{-} "Fn0";(110, -38) *++!U{y^{(p)}_{1,2}} *\cir<4pt>{} = "Fn2";
	\ar@{-} "Fn1";"Fn2";
	\ar@{-} "Fn0";(130, -38) *++!U{y^{(p)}_{t_{p},1}} *\cir<4pt>{} = "Fnn1";
	\ar@{} (0,0);(54,-30) *{\text{$y^{(2)}_{1,1}$}};
	\ar@{-} "Fn0";(135, -36) *++!L{y^{(p)}_{t_{p},2}} *\cir<4pt>{} = "Fnn2";
	\ar@{-} "Fnn1";"Fnn2";
	\ar@{} (0,0);(30,-8) *{\text{$w_{1}$}};
	\ar@{} (0,0);(70,-8) *{\text{$w_{2}$}};
	\ar@{} (0,0);(95,-12) *++!D{\cdots};
	\ar@{} (0,0);(120,-8) *{\text{$w_{p}$}};
	\ar@{} (0,0);(95,-24) *++!D{\cdots};
	\ar@{} (0,0);(30,-30) *++!D{\cdots};
	\ar@{} (0,0);(70,-30) *++!D{\cdots};
	\ar@{} (0,0);(120,-30) *++!D{\cdots};
\end{xy}

\bigskip

  \caption{Cameron-Walker graph}
  \label{fig:CameronWalkerGraph}
\end{figure}

Some invariants of Cameron-Walker graphs were computed in \cite{HKMT}:

\begin{theorem}\label{CWformulas}
  Let $G$ be a Cameron-Walker graph with notation
  as in Figure \ref{fig:CameronWalkerGraph}. 
  Then  
  \begin{enumerate}
    \item $\displaystyle |V(G)| = m + p + \sum_{i=1}^{m} s_{i} + 2 \sum_{j=1}^{p} t_{j}$;
    \item $\displaystyle \deg h_{R/I(G)}(t) 
      = \dim R/I(G) = \sum_{i=1}^m s_i + \sum_{j=1}^{p} \max\{t_j,1\}$; and
    \item $\displaystyle {\rm reg}(R/I(G)) = m + \sum_{j=1}^{p} t_j$. 
\end{enumerate}
\end{theorem}

\begin{proof}
  $(1)$ follows from the definition of a Cameron-Walker graph. See
  \cite[Proposition 1.3]{HKMT} for $(2)$. Statement $(3)$ is easy
  to see by computing ${\rm im}(G)$.
\end{proof}

The following class of Cameron-Walker graphs plays an important role 
in Section \ref{sec:CW}. 
\begin{construction}
  \label{const:CWabc}
 Fix $a,b \geq 1$ and $0 \leq c \leq b$.  Let $G = G_{a,b,c}$ be the Cameron-Walker
  graph whose bipartite part is the complete bipartite graph $K_{a,b}$, and 
  $s_1 = \cdots = s_a = 1$, $t_1=\cdots = t_c = 1$, 
  and $t_{c+1} = \cdots = t_b = 0$ (see Figure \ref{fig:Construction 2.4}). 
  
  \begin{figure}[htbp]
  \centering
\begin{xy}
	\ar@{} (0,0);(20, 12)  = "A";
	\ar@{.} "A";(130, 12)  = "B";
	\ar@{.} "A";(20, -12)  = "C";
	\ar@{.} "B";(130, -12)  = "D";
	\ar@{.} "C";"D";
	\ar@{} (0,0);(75,0) *{\text{$K_{a, b}$ \bf{on} $\{v_{1}, \ldots, v_{a}\} \cup \{w_{1}, \ldots, w_{b}\}$}};
	\ar@{} "A";(35, 13.5)  *\cir<4pt>{} = "E1";
	\ar@{-} "E1";(35, 24)  *\cir<4pt>{};
	\ar@{} "A";(65, 13.5)  *\cir<4pt>{} = "E2";
	\ar@{-} "E2";(65, 24)  *\cir<4pt>{};
	\ar@{} "A";(115, 13.5)  *\cir<4pt>{} = "E";
	\ar@{-} "E";(115, 24)  *\cir<4pt>{};
	\ar@{} (0,0);(35,8) *{\text{$v_{1}$}};
	\ar@{} (0,0);(65,8) *{\text{$v_{2}$}};
	\ar@{} (0,0);(90,4) *++!D{\cdots};
        \ar@{} (0,0);(115,8) *{\text{$v_{a}$}};
	\ar@{} (0,0);(90,16) *++!D{\cdots};
	\ar@{} "E1";(30, -13.5)  *\cir<4pt>{} = "F10";
	\ar@{-} "F10";(25, -24)  *\cir<4pt>{} = "F11";
	\ar@{-} "F10";(35, -24)  *\cir<4pt>{} = "Ft2";
	\ar@{-} "F11";"Ft2";
	\ar@{} "E1";(70, -13.5)  *\cir<4pt>{} = "F20";
	\ar@{-} "F20";(65, -24) *\cir<4pt>{} = "F21";
	\ar@{-} "F20";(75, -24)  *\cir<4pt>{} = "Ftt2";
	\ar@{-} "F21";"Ftt2";
	\ar@{} "E1";(80, -13.5)  *\cir<4pt>{}; 
	\ar@{} "E1";(120, -13.5)  *\cir<4pt>{} = "Fn0";
	\ar@{} (0,0);(30,-8) *{\text{$w_{1}$}};
	\ar@{} (0,0);(50,-12) *++!D{\cdots};
	\ar@{} (0,0);(70,-8) *{\text{$w_{c}$}};
	\ar@{} (0,0);(80,-8) *{\text{$w_{c + 1}$}};
	\ar@{} (0,0);(100,-12) *++!D{\cdots};
	\ar@{} (0,0);(120,-8) *{\text{$w_{b}$}};
	\ar@{} (0,0);(50,-24) *++!D{\cdots};
\end{xy}

\bigskip

  \caption{The Cameron--Walker graph $G_{a,b,c}$}
  \label{fig:Construction 2.4}
\end{figure}
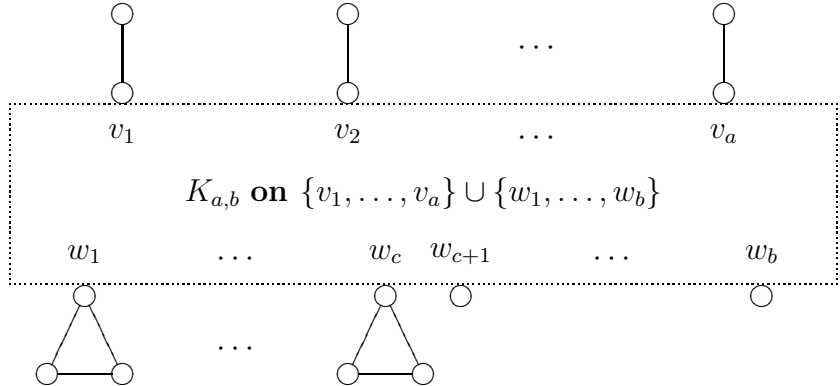
\end{construction}

\begin{example}
Let $a = 2, b = 3$, and $c = 2$. Then the graph $G_{2, 3, 2}$ is 
as in Figure \ref{fig:G2,3,2}. 
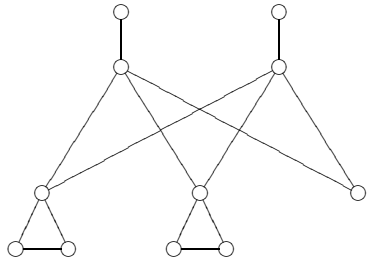
\begin{figure}[htbp]
  \hspace*{-4cm}  
  \centering
  \scalebox{.7}{
  \begin{xy}
	\ar@{} (0,0);(20, 12)  = "A";
	\ar@{} "A";(65, 13.5)  *\cir<4pt>{} = "E1";
	\ar@{-} "E1";(65, 24)  *\cir<4pt>{};
	\ar@{} "A";(95, 13.5)  *\cir<4pt>{} = "E2";
	\ar@{-} "E2";(95, 24)  *\cir<4pt>{};
	\ar@{-} "E1";(50, -10.5)  *\cir<4pt>{} = "F10";
	\ar@{-} "F10";(45, -21)  *\cir<4pt>{} = "F11";
	\ar@{-} "F10";(55, -21)  *\cir<4pt>{} = "Ft2";
	\ar@{-} "F11";"Ft2";
	\ar@{-} "E1";(80, -10.5)  *\cir<4pt>{} = "F20";
	\ar@{-} "F20";(75, -21) *\cir<4pt>{} = "F21";
	\ar@{-} "F20";(85, -21)  *\cir<4pt>{} = "Ftt2";
	\ar@{-} "F21";"Ftt2";
	\ar@{-} "E2";"F20";
	\ar@{-} "E2";"F10";
	\ar@{-} "E1";(110, -10.5)  *\cir<4pt>{} = "F"; 
	\ar@{-} "E2"; "F";
\end{xy}}
\caption{The Cameron--Walker graph $G_{2,3,2}$}\label{fig:G2,3,2}
\end{figure}

  \end{example}

As a direct application of Theorem \ref{CWformulas}, we
can compute some invariants of $G_{a,b,c}$. 
\begin{lemma}\label{specialfamily}
  Let $G = G_{a,b,c}$ be the Cameron-Walker
  graph as in Construction \ref{const:CWabc}.
  Then ${\rm reg}(R/I(G)) = a+c$ and $\deg h_{R/I(G)}(t) = a+b$.
\end{lemma}


\section{The regularity and $h$-polynomials of  Cameron-Walker graphs}
\label{sec:CW}

In this section, we prove our second main result, namely,
a characterization of the lattice points of $CW_{\mathcal{RD}} (n)$.
We then use this characterization to compute $|CW_{\mathcal{RD}}(n)|$.

\begin{theorem}\label{maintheorem1} For all $n \geq 5$,
  $(r,d) \in
  CW_{\mathcal{RD}}(n)$ if and only
  if
  \begin{enumerate}
  \item[$\bullet$] $2 \leq r \leq \lfloor \frac{n-1}{2} \rfloor$,
  \item[$\bullet$] $r \leq d \leq n -r $, and
  \item[$\bullet$] $d \geq -2r+n+1$.
    \end{enumerate}
\end{theorem}

\begin{proof}
  The hypothesis $n \geq 5$ allows us to
  assume the conditions are not vacuous.

  Suppose $(r,d)$ satisfy all the above conditions.
  Let $G = G_{d + 2r - n, n - 2r, n - r - d}$ be the graph of
  Construction \ref{const:CWabc} and $R = K[V(G)]$. The conditions
  on $(r,d)$ imply $d+2r-n, n-2r \geq 1$ and $0 \leq n-r-d \leq n-2r$,
  so the graph $G$ is defined.
  Then $|V(G)| = n$ and Lemma \ref{specialfamily} says that 
  \begin{itemize}
    \item ${\rm reg}(R/I(G)) = (d + 2r - n) + (n - r - d) = r$, 
    \item $\deg h_{R/I(G)}(t) = (d + 2r - n) + (n - 2r) = d$. 
  \end{itemize}
  Thus one has $(r,d) \in CW_{\mathcal{RD}}(n)$. 
  
  We will now verify that all the $(r,d) \in CW_{\mathcal{RD}}(n)$ satisfy
  the given inequalities. We know that $r+d \leq n$ (which is equivalent to
  $d \leq n-r$) holds for all graphs by \cite[Theorem 13]{HMVT}
  (also see \eqref{generalbound}).
  For Cameron-Walker
  graphs, it was shown that $d \geq r$ in \cite[Theorem 3.1]{HKMT}.
  Consequently,
  $r \leq d \leq n-r$, as desired.

  We now show that $r \geq 2$ for any Cameron-Walker graph.
  Suppose that $r =1$.  Then by Theorem \ref{CWformulas} (3),
  we must have $m=1$ and $t_j =0$ for all $j$. But this then forces
  the graph to be the star graph $K_{1,n-1}$,
  which is not considered as a Cameron-Walker graph.  So $r \geq 2$.

  To show that $r \leq \lfloor \frac{n-1}{2} \rfloor$,
  it suffices to show that $r < \frac{n}{2}$ (if $n$ is even
  $\lfloor \frac{n-1}{2} \rfloor = \frac{n}{2}-1$, and if $n$ is odd,
  $\lfloor \frac{n-1}{2} \rfloor =
  \lfloor \frac{n}{2} \rfloor$).  Suppose for a contradiction that
  $r \geq \frac{n}{2}$.
  Since $n = m+p+2\sum_{j=1}^p t_j +\sum_{i=1}^ms_i$ and
  $\sum_{i=1}^m {s_i} \geq m$, we have $n \geq 2m+2\sum_{j=1}^p t_j+p$.  Thus
  \[r \geq \frac{n}{2} \geq m+\sum_{j=1}^p t_j + \frac{p}{2} > r\]
  where the last inequality follows from
  Theorem \ref{CWformulas} (3).  This gives the
  desired contradiction.  This paragraph and the previous paragraph now
  show $2 \leq r \leq \lfloor \frac{n-1}{2} \rfloor$.

  Finally, we show that $d \geq -2r + n+1$.  We first note
  that we can rewrite $d$
  as
  \[d = \sum_{i=1}^m s_i + \sum_{j=1}^p \max\{t_j,1\} = \sum_{i=1}^m s_i + \sum_{j=1}^p t_j +
  |\{j ~|~ t_j = 0\}|.\]
  We then have
  \begin{eqnarray*}
  & & d+2r-n-1 \\
   & = & \sum_{i=1}^m s_i + \sum_{j=1}^p t_j +
  |\{j ~|~ t_j = 0\}| + 2\left(m+\sum_{i=1}^p t_j\right) - \left(m+\sum_{i=1}^m s_i + p + 2\sum_{j=1}^p t_j\right) -1 \\
  &=& \left(\sum^p_{i=1} t_j +  |\{j ~|~ t_j = 0\}|  - p\right) +(m -1) \geq 0
  \end{eqnarray*}
  because    $\sum^p_{i=1} t_j +  |\{j ~|~ t_j = 0\}| \geq p$ and $m \geq 1$.  Thus
  we have $d \geq -2r+n+1$, as desired.
  \end{proof}

When $(r,d) \in CW_\mathcal{RD}(n)$, we have $r+d \leq n$ by \eqref{generalbound}
and so $r+d = n-e$ for some integer $e \geq 0$.
As an interesting consequence, the following theorem
gives a graph theoretical interpretation of this integer $e$.

\begin{theorem}
  Suppose that $G$ is a Cameron-Walker graph on $n$ vertices
  with $(r,d) = ({\rm reg}(R/I(G)),\deg h_{R/I(G)}(t))$.
  If $r+d = n-e$, then $G$ has at least $e$ pendant triangles.
In particular, if $r+d = n$, then $G$ has no pendant triangles.
\end{theorem}

\begin{proof}
  We have
  \begin{eqnarray*}
    e & = & n-r-d \\
    &= &\left(m+\sum_{i=1}^m s_i + p + 2\sum_{j=1}^p t_j\right) - \left(m + \sum_{j=1}^p t_j\right) -  \left(\sum_{i=1}^m s_i + \sum_{j=1}^p t_j +
    |\{j ~|~ t_j = 0\}|\right) \\
    &=& p- |\{j ~|~ t_j = 0\}|.
  \end{eqnarray*}
  So $e$ is the number of $j \in \{ 1, \ldots, p \}$ with $t_j \geq 1$,
  i.e., the vertices $w_j \in \{w_1,\ldots,w_p\}$
  that have a pendant triangle attached to it. 
  So, $e$ is a lower bound on the number
  of pendant triangles in $G$.

  \par
   Moreover if $e=0$, then 
   $p =  |\{j ~|~ t_j = 0\}|$.
   This means that $G$ has no pendant triangles. 
\end{proof}

Since $CW_{\mathcal{RD}}(n)$ is a lattice polytope, it is natural
to ask how many integer points are in this lattice.
Using Theorem \ref{maintheorem1} we can answer this question.

\begin{theorem}\label{|CW_{RD}(n)|}
  Fix an integer $n \geq 5$, and let $g = \lfloor \frac{n+1}{3} \rfloor$ and
  $f = \frac{n+1}{3} - g$  (so $f = 0,\frac{1}{3}$, or $\frac{2}{3}$).
  Then
  \[|CW_{\mathcal{RD}}(n)| = 
  \begin{cases}
    \frac{3g^{2}}{4} - 1 = \frac{1}{12} (n+1)^2 -1
      & \mbox{if $g$ is even and $f=0$}, \\
    \frac{3(g^{2}-1)}{4} - 1 = \frac{1}{12} (n+1)^2 - \frac{7}{4}
      & \mbox {if $g$ is odd and $f=0$}, \\
    \frac{g(3g + 2)}{4} - 2 = 
    \frac{1}{12} (n+6)(n-4)
      & \mbox{if $g$ is even and $f=\frac{1}{3}$}, \\
    \frac{(3g-1)(g+1)}{4} - 1 = 
    \frac{1}{12} (n-3)(n+5)
      & \mbox {if $g$ is odd and $f=\frac{1}{3}$}, \\
    \frac{g(3g+4)}{4} - 1 = 
     \frac{1}{12} (n-3)(n+5)
      & \mbox{if $g$ is even and $f=\frac{2}{3}$}, \\
    \frac{3g^{2} + 4g - 3}{4} - 1 = 
    \frac{1}{12} (n+6)(n-4)
      & \mbox {if $g$ is odd and $f=\frac{2}{3}$.}
    \end{cases}\]
\end{theorem}

\begin{proof}
  By Theorem \ref{maintheorem1}, we have inequalities:
  \begin{displaymath}
    2 \leq r \leq \left\lfloor \frac{n-1}{2} \right\rfloor, \quad
    r \leq d \leq n-r,
    \quad d \geq -2r+n+1.
  \end{displaymath}
  We fix an integer $r$ with $2 \leq r \leq \lfloor (n-1)/2 \rfloor$. 
  When $r \leq n-2r+1$, namely $r \leq (n+1)/3$, 
  the number of $d$ satisfying $(r,d) \in CW_{\mathcal{RD}} (n)$ is $r$.
  Indeed, if $r \leq n-2r+1$, we have $n-2r+1 \leq d \leq n-r$, so
  there are $r$ possibilities for $d$.
  When $r > n-2r+1$, namely $r > (n+1)/3$, 
  the number of $d$ satisfying $(r,d) \in CW_{\mathcal{RD}} (n)$ is
  $n-2r+1$.  To see this, in this range, we must have
  $r \leq d \leq n-r$, so $d = r+i$ with $i=0,\ldots,n-2r$.
  Summing up $d$ for all $r$, we can compute $|CW_{\mathcal{RD}} (n)|$. 

  Note the number of lattice points will thus depend upon knowing
  the exact value of $\frac{n+1}{3}$.  In particular, from the previous
  paragraph
  \begin{equation}\label{cwformula}
  |CW_{\mathcal{RD}}(n)| = \sum_{r=2}^{\lfloor \frac{n+1}{3} \rfloor} r
    + \sum_{r = \lfloor \frac{n+1}{3} \rfloor +1}^{\lfloor \frac{n-1}{2} \rfloor}
    (n-2r+1).
    \end{equation}
    If $f=0$, then $g = \frac{n+1}{3}$,
    and consequently, $3g = n +1$.  Plugging this information into
    \eqref{cwformula}, we
    get
    \[|CW_{\mathcal{RD}}(n)| = \sum_{r=2}^g r + \sum_{r = g+1}^{\lfloor \frac{3g-2}{2}
      \rfloor} (3g-2r).\]
    If $g$ is even, then  $\lfloor \frac{3g-2}{2} \rfloor = \frac{3g-2}{2}$, and
    consequently,
    \[|CW_{\mathcal{RD}}(n)| = 2+3+\cdots + g + (g-2)+(g-4)+ \cdots + 2 =
    \frac{3g^{2}}{4} - 1.\]
      On the other hand, if $g$ is odd, then $\lfloor \frac{3g-2}{2} \rfloor =
      \frac{3g-3}{2}$, and consequently,
      \[|CW_{\mathcal{RD}}(n)| = 2+3+\cdots + g + (g-2)+(g-4)+ \cdots + 5+3 =
        \frac{3(g^{2}-1)}{4} - 1. \]
        Because $n=3g-1$, we can rewrite both expressions in terms of $n$
        and derive the stated formulas.
        
        The other cases are computed in a similar fashion,
        so we have omitted the details.
\end{proof}

The next result is an immediate corollary of Theorem \ref{|CW_{RD}(n)|}.

\begin{corollary}
  \[\lim_{n \rightarrow \infty} \frac{|CW_\mathcal{RD}(n)|}{n^2} = \frac{1}{12}.\]
  \end{corollary}


\section{Future directions}

We conclude this paper with a question inspired by the results
of this paper. It would be interesting to compare the number of
integer points in $CW_{\mathcal{RD}}(n)$ to the number of integer points
in $\mathcal{RD}(n)$.  In particular, one might
wish to know what percentage of possible $(r,d) = ({\rm reg}(R/I(G)),
\deg h_{R/I(G)}(t))$ can be realized by Cameron-Walker graphs.
Thus, an answer to the following question would be of interest:
\begin{question}\label{limitquestion}
  What is the value of
  \[\lim_{n \rightarrow \infty} \frac{|CW_{\mathcal{RD}}(n)|}{|\mathcal{RD}(n)|}? \]
  \end{question}

It is not clear that this limit exists due, in part, to the fact
that we can only bound $|\mathcal{RD}(n)|$ (see Theorem \ref{maintheorem}).
Observe that to show that this limit exists, it is enough
  to show that $\frac{|CW_{\mathcal{RD}}(n)|}{|\mathcal{RD}(n)|} \leq
  \frac{|CW_{\mathcal{RD}}(n+1)|}{|\mathcal{RD}(n+1)|}$ for all $n$ since
  $\frac{|CW_{\mathcal{RD}}(n)|}{|\mathcal{RD}(n)|} \leq 1$,
  and then one can use the fact that we have a
  bounded monotic increasing sequence.

  If we assume that the limit exists, we can give a partial answer
  to Question \ref{limitquestion}.

\begin{theorem}  Suppose that $\lim_{n \rightarrow \infty} \frac{|CW_{\mathcal{RD}}(n)|}{|\mathcal{RD}(n)|}$ exists.  Then
 \[ \frac{2}{9} \leq \lim_{n \rightarrow \infty} \frac{|CW_{\mathcal{RD}}(n)|}{|\mathcal{RD}(n)|} \leq \frac{1}{3}.\]
\end{theorem}

\begin{proof}
  Note that by Theorem \ref{|CW_{RD}(n)|}, we always have
  $|CW_{\mathcal{RD}}(n)| =\frac{1}{12}(n+a)(n+b) + c$ for some
  $a,b$ and $c$ that satisfy $-4 \leq a,b \leq 6$ and $-\frac{7}{4} \leq
  c \leq 0$.  Thus, for all $n \geq 5$,
  \[\frac{1}{12}(n-4)(n-4)-\frac{7}{4} \leq 
|CW_{\mathcal{RD}}(n)| \leq \frac{1}{12}(n+6)(n+6).\]

 Using the fact that if $(r,d) \in \mathcal{RD}(n)$, then $r+d \leq n$ and $1 \leq r \leq
  \lfloor \frac{n-1}{2} \rfloor$, we get an upper bound
  \[|\mathcal{RD}(n)| \leq \binom{n}{2} -
  \binom{\lceil \frac{n+1}{2} \rceil}{2},\] 
  where we use the fact that $n-1 < \lfloor \frac{n-1}{2} \rfloor
  + \lceil \frac{n+1}{2} \rceil< n+1$.
  Combining this bound with the lower bound for $|CW_{RD}(n)|$ above gives
  
  \[\frac{|CW_{\mathcal{RD}}(n)|}{|\mathcal{RD}(n)|} \geq
  \frac{\frac{1}{12}(n-4)(n-4)-\frac{7}{4}}{ \binom{n}{2} - \binom{\lceil \frac{n+1}{2} \rceil}{2}}.\]
Letting $n \rightarrow \infty$ on the right hand side gives $\frac{2}{9}$.  

Moreover, by Theorem \ref{maintheorem},we get a lower bound
\begin{eqnarray*}
|\mathcal{RD}(n)| &\geq&  \left(\left\lfloor \frac{n}{2}  \right\rfloor\right)^{2} . 
\end{eqnarray*}
Hence we have the bound
\[
\frac{|CW_{\mathcal{RD}}(n)|}{|\mathcal{RD}(n)|} \leq
\frac{\frac{1}{12}(n+6)(n+6)}{\left(\left\lfloor \frac{n}{2}  \right\rfloor\right)^{2} }
.\]
Letting $n \rightarrow \infty$ on the right hand side gives $\frac{1}{3}$. 
\end{proof}



\bibliographystyle{plain}

\end{document}